\documentclass[10pt,oneside,a4paper]{amsart}
\usepackage{amsmath,amsfonts,amsthm, amssymb}
\usepackage{array}
\usepackage[all]{xy}
\usepackage{color}
\usepackage{xcolor}
\usepackage{graphicx}

\usepackage{hyperref}
\hypersetup{
    colorlinks=true,
    linkcolor=blue,
    filecolor=magenta,      
    urlcolor=cyan,
    pdftitle={Overleaf Example},
    pdfpagemode=FullScreen,
    }

\theoremstyle{plain}
\newtheorem{theorem}{Theorem}[section]
\newtheorem{proposition}[theorem]{Proposition}
\newtheorem{lemma}[theorem]{Lemma}
\newtheorem{corollary}[theorem]{Corollary}

\theoremstyle{definition}
\newtheorem{definition}[theorem]{Definition}

\theoremstyle{remark}

\newtheorem{remark}[theorem]{Remark}
\newtheorem{example}[theorem]{Example}

\renewcommand{\lim}{\mathrm{lim}}

\newcommand{\Aut}{\mathrm{Aut}}

\newcommand{\lra}{\longrightarrow}

\allowdisplaybreaks
\numberwithin{equation}{section}

\author{ Dinh Van Hoang}
\address[Dinh Van Hoang]{Faculty of Advanced Education, Ho Chi Minh City University of Technology and Engineering, Vietnam }
\email{hoangdv@hcmute.edu.vn}

\author{Phan Thanh Toan}
\address[Phan Thanh Toan]{Analytical and Algebraic Methods in Optimization Research Group, Faculty of Mathematics and Statistics, Ton Duc Thang University, Ho Chi Minh City, Vietnam}
\email{phanthanhtoan@tdtu.edu.vn}

\begin{document}

\title{Noether's normalization in iterated skew polynomial rings}

\begin{abstract}
The classical Noether Normalization Lemma  states that if $S$ is a finitely generated algebra over a field $k$, then there exist elements $x_1,\dots,x_n$ which are algebraically independent over $k$ such that $S$ is a finite module over $k[x_1,\dots,x_n]$. This lemma has been studied intensively in different flavors. In 2024, Elad Paran and Thieu N. Vo successfully generalized this lemma for the case when $S$ is a quotient ring of the skew polynomial ring $D[x_1,\dots,x_n;\sigma_1,\dots,\sigma_n]$. In this paper, we investigate this lemma in a more general setting when $S$ is a quotient ring of an iterated skew polynomial ring $D[x_1;\sigma_1,\delta_1]\dots[x_n;\sigma_n,\delta_n]$. We extend several key results of Elad Paran and Thieu N. Vo  to this broader context and introduce a new version of Combinatorial Nullstellensatz over division rings.
\end{abstract}


\maketitle

\section{Introduction}
The Noether Normalization Lemma is a fundamental result in algebraic geometry and commutative algebra. It provides a means of simplifying the structure of finitely generated algebras. More precisely, it asserts that for any finitely generated algebra $S$ over a field $k$, one can find algebraically independent elements $x_1,\dots,x_n$ over $k$ such that $S$ is a finite module over $k[x_1,\dots, x_n]$. Moreover, it is shown that $n$ is also the Krull dimension of $S$. 
This lemma plays a crucial role in dimension theory and has far-reaching applications across algebraic geometry, invariant theory, and commutative algebra, for example, see \cite{Ketan, Gleb, Isabel}, etc..   

Extending such foundational results from algebraic geometry to the noncommutative setting has been a central theme in modern algebra. When we replace the ground field $k$ by a division ring $D$, there exist numerous methods to construct extensions of $D$ based on finitely many generators $x_1,\dots,x_n$. A natural question is to determine for which classes of such extensions an analogue of Noether normalization continues to hold. This question has been investigated in several recent works. In particular, Elad Paran and Thieu N. Vo \cite{Thieu1} introduced the notion {\em``automorphically normalizable"} and established that Noether-type normalization holds for skew polynomial rings of the form 
$S=D[x_1,\dots,x_n;\sigma, \ldots, \sigma]$. 

Motivated by these developments, the aim of this paper is to further explore Noether normalization in the setting of more general noncommutative algebras, namely iterated skew polynomial rings of the form $D[x_1,\dots,x_n;(\omega_1,\delta_1),\dots, (\omega_n,\delta_n)]$ in which $\omega_1,\dots,\omega_n\in \mathrm{Aut}(D)$ and $\delta_i$ is an $\omega_i$-derivation of $D$ for $i=1,\dots,n$.  
In the first step, we introduce the notion \textit{``skew  normalizable"}: 
\begin{definition}\label{def1}
	Let $D\subseteq S$ be a ring extension, where $D$ is a division ring.
	\begin{enumerate}
		\item $S$ is said to be  \textit{skew  normalizable} over $D$ if there is a subring $D\subseteq R\subseteq S$ such that $R$ is isomorphic to $D[x_1,\dots,x_n;(\omega_1,\delta_1),\dots, (\omega_n,\delta_n)]$  for some $\omega_1,\dots,\omega_n\in \mathrm{Aut}(D)$, $\delta_i$ is an $\omega_i$-derivation of $D$ for $i=1,\dots,n$, and $S$ is a left finite module over $R$.
		\item The tuple $(\omega_1,\delta_1;\dots;\omega_n,\delta_n)$ is called \textit{skew  normalizable} over $D$ if every quotient ring of $D[x_1,\dots,x_n;(\omega_1,\delta_1),\dots, (\omega_n,\delta_n)]$ is skew normalizable over $D$.
	\end{enumerate}
\end{definition}

Note that if we choose $\delta_1=\dots=\delta_n=0$, then our definition is reduced to the definition of ``automorphically normalizable" in \cite[Definition 1.2]{Thieu1}. 

In the next step, we study when the tuple $(\omega_1,\delta_1;\dots;\omega_n,\delta_n)$ is skew  normalizable. In the case that there is no derivative and $\omega_1=\dots=\omega_n$, Elad Paran and Thieu N. Vo obtained the important result \cite[Theorem 3.7]{Thieu1}, which is stated as follows.
\begin{theorem}{\label{Thieu_Elad_main}}
	Let $D$ be a division algebra and let $\sigma$ be an automorphism of $D$ such that the field
	$F = Z(D)\cap D_\sigma$  is infinite. Then any constant tuple $(\sigma, \dots , \sigma)$ is automorphically normalizable over $D$. Equivalently, every quotient ring of the form $D[x_1,\dots,x_n;\sigma,\dots,\sigma]/I$, where $I$ is a two-sided ideal, is automorphically normalizable.
\end{theorem}

The main result established in our paper is a generalization of Theorem  \ref{Thieu_Elad_main}. Our result is stated as follows. 
\begin{theorem}
	Let $D$ be a division ring, $\omega$ be an automophism of $D$,  and $\delta_1,\dots,\delta_n$ be $\omega$-derivations of $D$ such that $\omega\circ\delta_i=\delta_i\circ\omega$ and $\delta_i\circ\delta_j=\delta_j\circ\delta_i$ for all $1\le i,j\le n$. Assume that $F=Z(D)\cap D_{\omega}\cap \mathrm{Ker}(\delta_1)\cap\cdots\cap \mathrm{Ker}(\delta_n)$ is an infinite set. Then the tuple $(\omega,\delta_1;\dots;\omega,\delta_n)$ is skew normalizable.
\end{theorem}

One of the key steps in the proof of our theorem relies
on the idea of the Combinatorial Nullstelensatz given by Noga Alon in \cite{NAlon1} (another proof is given by Michalek in \cite{Michalek}), which was generalized to the case of division rings by Elad Paran in \cite{Paran1} as follows.

\begin{theorem}[Combinatorial Nullstellensatz for division rings II]
	Let $D$ be a division ring and $p\in D[x_1,\dots,x_n]$ be a polynomial of degree $\sum_{i=1}^nk_i$, in which the coefficient of $x_1^{k_1}\dots x_n^{k_n}$ is non-zero. Let $A_1,\dots,A_n$ be algebraic subsets of $D$ such that $A_1\times\dots \times A_n\subset D_c^n$ and $|A_i|>k_i$ for $i=1,\dots,n$. Then, there is a point in $A_1\times\dots\times A_n$ at which $p$ does not vanish.
\end{theorem}
In this paper, we present a new version of this theorem in which each set $A_i$ is not required to be algebraic, and the condition that the elements in $A_i$ commute with the elements in $A_j$ is removed. Instead, we only require that the elements of each $A_i$ are pairwise non-conjugate.

\begin{theorem}
	[Combinatorial Nullstellensatz over division rings III]\label{ComNulstellensatz} Let $D$ be a division ring and $$\displaystyle f(t_1,\dots,t_n)=\sum_{I=(i_1,\dots,i_n)\in \mathbb{N}^n}b_I t_1^{i_1}\dots t_n^{i_n}$$ be a   polynomial of degree $m>0$ in $D[t_1,\dots,t_n]$.
	Let $A_1,\dots, A_n$ be subsets of  $D$ such that the elements of each $A_i$ are pairwise non-conjugate and $|A_i|>m$ for $i=1,\dots,n$. Then there exists $(a_1,\dots,a_n)\in A_1\times\dots\times A_n$ such that $$f(a_1,\dots,a_n)= \sum_{I=(i_1,\dots,i_n)\in \mathbb{N}^n}b_I a_1^{i_1}\dots a_n^{i_n}\ne 0.$$
\end{theorem}


\section{Preliminaries}
In this section, we present basic properties of the iterated skew polynomial ring $R[t_1;\omega_1,\delta_1] \dots[t_n;\omega_n,\delta_n]$ and the ring $R[t_1,\dots,t_n; (\omega_1,\delta_1),\dots,(\omega_n,\delta_n)]$. Standard references for this part include  \cite{Goodearl} and \cite{Ore}. This topic has been studied intensively from both theoretical point of view and applied one, for example, see \cite{Paran2, Boucher, Martinez, Goodearl2,Taelman}. Throughout this paper, all rings are assumed to be associative with unity, by a division ring we mean a ring in which every nonzero element has a multiplicative inverse.

\begin{definition}
	Let $R$ be a ring and $\omega$ be an endomorphism of $R$. An additive map $\delta:R\lra R$ is called an $\omega$-derivation of $R$ if  $\delta(ab)=\omega(a)\delta(b)+\delta(a)b,\ \forall a,b\in R$.
\end{definition}

The following two definitions are based on Chapter 2 of \cite{Goodearl}.
\begin{definition}\label{skew_poly_ring1} Let $R \subseteq S$ be a ring extension and $t_1,\dots,t_n$ be indeterminates.
	\begin{enumerate}
		\item  Let $\omega\in\mathrm{Aut}(R)$ and $\delta$ be an $\omega$-derivation  of $R$, an element $s\in S$ is said to be {\em automorphic} over $R$ with respect to $(\omega,\delta)$ if $$sr=\omega(r)s+\delta(r), \text{ for all } r\in R.$$ 
		\item Let $\omega_1\in\mathrm{Aut}(R)$ and  $\delta_1$ be an $\omega_1$-derivation  of $R$, the {\em skew polynomial ring} $R[t_1;\omega_1,\delta_1] $ is the ring containing $R$ as a subring such that $R[t_1;\omega_1,\delta_1] $ is a free left module over $R$ with basis $1,t_1,t_1^2,\dots$, and $$t_1r=\omega_1(r)t_1+\delta_1(r), \text{ for all } r\in R.$$
		\item The {\em iterated skew polynomial ring} is defined by $$R[t_1;\omega_1,\delta_1] \dots[t_n;\omega_n,\delta_n] = R_n$$
        where $R_1=R[t_1;\omega_1,\delta_1]$, and inductively, for $i=2,\dots,n$, 
        $$R_i=R_{i-1}[t_i;\omega_i,\delta_i]$$ in which $\omega_i$ is an automorphism of $R_{i-1}$ and $\delta_i$ is an $\omega_i$-derivation of $R_{i-1}$. Explicitly, the condition $$t_if=\omega_i(f)t_i+\delta_{i}(f),\ \forall f\in R_{i-1},$$
		is equivalent to
	$$	\begin{cases}
			 t_ir=\omega_i(r)t_i+\delta_{i}(r) \text{ for all  } r\in R \\
			 t_it_j=\omega_i(t_j)t_i+\delta_{i}(t_j) \text{ for all } 1\le j<i
		\end{cases}.$$ 
		
		\item  Elements $s_1,\dots,s_n$ in $S$ are said to form an {\em iterated automorphic sequence over} $R$ with respect to $(\omega_1,\delta_1),\dots,(\omega_n,\delta_n)$ if for $i$ from 1 to $n$, inductively, $\omega_i$ is an automorphism of $R_{i-1}= R[s_1;\omega_1,\delta_1]\dots[s_{i-1};\omega_{i-1},\delta_{i-1}]$, $\delta_i$ is an $\omega_i$-derivation of $R_{i-1}$, and $$s_if=\omega_i(f)s_i+\delta_{i}(f),\ \forall f\in R_{i-1}.$$
	\end{enumerate} 
\end{definition}

\begin{definition}\label{skew_poly_ring2}
	Let $R$ be a ring, and let $\omega_1,\dots,\omega_n\in\Aut(R)$. For $i=1,\dots,n$,  let $\delta_i$ be an $\omega_i$-derivation of $R$. The ring of skew polynomials $R[t_1,\dots,t_n; (\omega_1,\delta_1),\dots,$ $(\omega_n,\delta_n)]$ is defined to be the ring consisting of $R$ and {commuting} variables $t_1,\dots, t_n$, where the multiplication subjects to the relations
	$$t_ir=\omega_i(r)t_i+\delta_i(r),\ \text{ for all } r\in R,\ i=1,\dots n.$$
\end{definition}

\begin{remark}\label{Convertt}
	We can convert the skew polynomial ring $R[t_1,\dots,t_n; (\omega_1,\delta_1),\dots,(\omega_n,\delta_n)]$ into the iterated skew polynomial ring $R[t_1;\bar{\omega}_1,\bar{\delta}_1]\dots[t_n;\bar{\omega}_n,\bar{\delta}_n]$ by setting $\bar{\omega}_1=\omega_1$, $\bar{\delta}_1=\delta_1$, for $i=2,\dots,n$, we set
   $$ \left\{\begin{matrix}
 \bar{\omega}_i(r)=\omega_i(r)& \text{ for all } r\in R\\
 \bar{\omega}_i(t_j)=t_j& \text{ for } j=1,\dots,i-1 \\
\end{matrix}\right.$$

and    $$ \left\{\begin{matrix}
 \bar{\delta}_i(r)=\delta_i(r)& \text{ for all } r\in R\\
 \bar{\delta}_i(t_j)=0& \text{ for } j=1,\dots,i-1. \\
\end{matrix}\right.$$
    
\end{remark}

\medskip
The following lemma is well known.
\begin{lemma}\label{leftbasis}
Let $R\subseteq S \subseteq T$ be a sequence of ring extensions. Assume that $S$ is a free left module over $R$ with basis $\{u_i: i\in \mathbb{N}\}$, and that $T$ is a left free module over $S$ with basis $\{v_j: j\in \mathbb{N}\}$. Then $T$ is a left free module over $R$ with basis $\{u_iv_j: i,j\in\mathbb{N}\}$.
\end{lemma}

The following lemma is necessary for subsequent computations.
\begin{lemma}\label{calculation1} Let $S=R[t_1;\omega_1,\delta_1]\dots[t_n;\omega_n,\delta_n]$ be an iterated skew polynomial ring over a ring $R$. Then 
	\begin{enumerate}
		\item  For $k\ge 1$, $$t_i^kr=\omega_i^k(r)t_i^k+a_{k-1}t_i^{k-1}+\dots+a_1t_i+\delta_i^k(r)$$ for some elements $a_1,\dots, a_{k-1}\in R$.
		\item $$t_1^{i_1}\dots t_n^{i_n}r=\omega_1^{i_1}\circ\cdots\circ\omega_n^{i_n}(r)t_1^{i_1}\dots t_n^{i_n}+ \sum_{K=(k_1,\dots,k_n)\in \mathbb{N}^n}a_Kt_1^{k_1}\dots t_n^{k_n} $$
		where $a_K\in R$ and the summation is taken over all $K = (k_1, \ldots, k_n) \in \mathbb{N}^n$ satisfying $0\le k_j\le i_j$, and $k_1+\dots+k_n<i_1+\dots+i_n$.
		\item For $r\in R$, $$(rt_j)^m=r\omega_j(r)\dots\omega_j^{m-1}(r)t_j^m+ a_{m-1}t_j^{m-1}+\dots+a_1t_j$$ for some $a_1,\dots,a_{m-1}\in R$.
		\item $S$ is a free left module over $R$ with basis $\{t_1^{i_1}\dots t_n^{i_n}: i_1,\dots, i_n\ge 0\}.$
	\end{enumerate}
\end{lemma}

\begin{proof}
	

(1) The equation holds for $k=1$. For $k>1$, assume that the equation holds up to $k-1$: $$t_i^{j}r=\omega_i^{j}(r)t_i^j+a_{j-1}t_i^{j-1}+\dots+a_1t_i+\delta_i^{j}(r)$$ for  $j\le k-1$, where $a_1,\dots, a_{j-1}\in R$. Then we have 
	\begin{align*}
		t_i^{k}r&=t_i(t_i^{k-1}r)\\
		&=t_i\big(\omega_i^{k-1}(r)t_i^{k-1}+a_{k-2}t_i^{k-2}+\dots+a_1t_i+\delta_i^{k-1}(r)\big) \text { for some } a_1,\dots, a_{k-2}\in R.\\
		&=[\omega_i\omega_i^{k-1}(r)t_i+\delta_i\omega_i^{k-1}(r)]t_i^{k-1}+[\omega_i(a_{k-2})t_i+\delta_i(a_{k-2})]t_i^{k-2}+\cdots\\
		& \ \ \ \ \cdots +[\omega_i(a_1)t_i+\delta_i(a_1)]t_i+\omega_i\delta_i^{k-1}(r)t_i+\delta_i\delta_i^{k-1}(r)\\
		&=\omega_i^k(r)t_i^k+b_{k-1}t_i^{k-1}+\dots+b_1t_i+\delta_i^k(r),
	\end{align*}
where $b_{k-1}=\delta_i\omega_i^{k-1}(r)+\omega_i(a_{k-2})\in R$ and $b_l=\delta_i(a_l)+\omega_i(a_{l-1})\in R$ for $1\le l \le k-2$.

\medskip

(2) When $n=1$, the equation holds true by (1). For $n>1$, assume that the equation holds true up to $n-1$. We now prove that the equation holds true for $n$. By induction hypothesis, we have
\begin{align*}
	t_1^{i_1}\dots t_n^{i_n}r&=t_1^{i_1}(t_2^{i_2}\cdots t_n^{i_n}r)
	\\&=t_1^{i_1}\bigg(\omega_2^{i_2}\circ\cdots\circ\omega_n^{i_n}(r)t_2^{i_2}\dots t_n^{i_n}+ \sum_{\substack{K=(k_2,\dots,k_n)\in \mathbb{N}^n, k_j\le i_j\\ k_2+\cdots+k_n<i_2+\dots+i_n}}a_Kt_2^{k_2}\dots t_n^{k_n}\bigg), \\
	&\ \ \ \ \text { where } a_K\in R\\
	&=\big[\omega_1^{i_1}\circ \cdots \circ\omega_n^{i_n}(r)t_1^{i_1}+\sum_{j=1}^{i_1}a_jt_{i_1}^{i_1-j}\big]t_2^{i_2}\dots t_n^{i_n}\\
	&\ \ \ \ \ +\sum_{\substack{K=(k_2,\dots,k_n)\in \mathbb{N}^n, k_j \le i_j\\ k_2+\cdots+k_n<i_2+\dots+i_n}}\big[\omega_1^{i_1}(a_K)t_1^{i_1}+\sum_{j=1}^{i_1}a_{K,j}t_1^{i_1-j}\big]t_2^{k_2}\dots t_n^{k_n},\\	
    &\ \ \ \ \text { where } a_j\in R\\
    &=\omega_1^{i_1}\circ\cdots \circ\omega_n^{i_n}(r)t_1^{i_1}\dots t_n^{i_n}+ \sum_{\substack{K=(k_1,\dots,k_n)\in \mathbb{N}^n, k_j \le i_j\\ k_1+\cdots+k_n<i_1+\dots+i_n}}b_Kt_1^{k_1}\dots t_n^{k_n},\\
    &\ \ \ \ \text { where } b_K\in R.
\end{align*}

(3) The equation can be proved by induction on $m$ as in (1).

\medskip

(4) We prove the claim by induction on $n$. The case $n=1$ is obvious. Let $R'=R[t_1;\omega_1,\delta_1]\dots[t_{n-1};\omega_{n-1},\delta_{n-1}]$. By induction hypothesis, $R'$ is a left module over $R$ with basis $\{t_1^{i_1}\dots t_{n-1}^{i_{n-1}}: i_1,\dots,i_{n-1}\in\mathbb{N}\}$. Now, by definition, we have $S=R'[t_n; \omega_n,\delta_n]$ is a left free module over $R'$ with basis $\{t_n^{i_n}:i_n\in\mathbb{N}\}$. By Lemma \ref{leftbasis}, $S$ is a left free module over $R$ with basis $\{t_1^{i_1}\dots t_n^{i_n}: i_1,\dots, i_n\in \mathbb{N}\}$.
\end{proof}

\begin{lemma}\label{goodearl1}
	Let $R\subseteq S$ be a ring extension, $\omega$ be an automorphism of $R$ and $\delta$ be an $\omega$-derivation of $R$. Assume that $s\in S$ is automorphic over $R$ with respect to $(\omega,\delta)$. Then there exists a unique homomorphism $\phi:R[x;\omega,\delta]\lra S$ such that $\phi(x)=s$ and $\phi(r)=r$ for all $r\in R$.
\end{lemma}
\begin{proof}
	Apply \cite[Proposition 2.4, page 37]{Goodearl}.
\end{proof}
\begin{proposition}\label{eval1}
	Let $R \subseteq S$ be a ring extension.  Suppose that $s_1,\dots,s_n\in S$ form an iterated automorphic sequence with respect to $(\omega_1,\delta_1),\dots,(\omega_n,\delta_n)$. Then there exists a unique ring homomorphism $\phi:R[x_1;\omega_1,\delta_1]\dots[x_n;\omega_n,\delta_n]\lra S$ such that $\phi(x_i)=s_i$ for all $1\le i\le n$ and $\phi(r)=r$ for all $r\in R$.
\end{proposition}
\begin{proof}
	Let $T=R[x_1;\omega_1,\delta_1]\dots[x_{n-1};\omega_{n-1},\delta_{n-1}]$. By induction, there exists a unique homomorphism $\varphi:T\lra S$ such that $\varphi(x_i)=s_i$ for all $1\le i\le n-1$ and $\varphi(r)=r$ for all $r\in R$. Let $U=\varphi(T)=R[s_1;\omega_1,\delta_1]\dots[s_{n-1};\omega_{n-1},\delta_{n-1}]$. By Lemma \ref{goodearl1}, there exists a unique homomorphism $\varphi':U[x_n;\omega_n,\delta_n]\lra S$ such that $\varphi'(u)=u$ for all $u\in U$ and $\varphi'(x_n)=s_n$. On the other hand, the homomorphism $\varphi$ is extended naturally to $\bar{\varphi} :T[x_n; \omega_n,\delta_n]\lra U[x_n; \omega_n,\delta_n]$. Our proof is completed by setting $\phi=\varphi'\circ\bar{\varphi}$.
\end{proof}

\begin{proposition}\label{eval2}
Let $R\subseteq S$ be a ring extension. Let $\omega_1,\dots,\omega_n\in\Aut(R)$, and let $\delta_i$ be $\omega_i$-derivation of $R$ for $i=1,\dots,n$. Assume that $s_1,\dots,s_n \in S$ are commuting automorphic elements over $R$ with respect to $(\omega_1,\delta_1),$ $\dots,(\omega_n,\delta_n)$ respectively. Then there exists a unique ring homomorphism $$\phi:R[x_1,\dots,x_n; (\omega_1,\delta_1),\dots,(\omega_n,\delta_n)] \lra S$$  such that $\phi(r)=r, \forall r\in R$ and $\phi(x_i)=s_i,\ \forall i=1,\dots, n$. In other words, the evaluation map at $(s_1,\dots,s_n)$ $$\mathrm{Eval}_{(s_1,\dots,s_n)}:  R[x_1,\dots,x_n; (\omega_1,\delta_1),\dots,(\omega_n,\delta_n)] \lra S$$
where $\mathrm{Eval}_{(s_1,\dots,s_n)}(f):=f(s_1,\dots,s_n)$ is a ring homomorphism. In particular, $\mathrm{Eval}_{(s_1,\dots,s_n)}(r)=r$ for all $r\in R$.
\end{proposition}

\begin{proof}
By Remark \ref{Convertt}, we can convert the ring $R[x_1,\dots,x_n; (\omega_1,\delta_1),\dots,(\omega_n,\delta_n)]$ into the ring $R_n=R[t_1;\bar{\omega}_1,\bar{\delta}_1]\dots[t_n;\bar{\omega}_n,\bar{\delta}_n]$.
    It is easy to see that the identity map $\mathrm{id}:R[x_1,\dots,x_n; (\omega_1,\delta_1),\dots,(\omega_n,\delta_n)]\lra R_n$ is an isomorphism of rings. Moreover, $s_1,\dots,s_n$ is an iterated automorphic sequence with respect to $(\bar{\omega}_1,\bar{\delta}_1),\dots,(\bar{\omega}_n,\bar{\delta}_n)$ over $R$. Let $\phi: R_n\lra S$ be the homomorphism defined as in Proposition \ref{eval1}. For any skew polynomial $f$ in $R[x_1,\dots,x_n; (\omega_1,\delta_1),\dots,(\omega_n,\delta_n)]$, we define $\mathrm{Eval}_{(s_1,\dots,s_n)}(f)=\phi(f)=f(s_1,\dots,s_n)$,  thus $\mathrm{Eval}_{s_1,\dots,s_n}$ is a homomorphism of rings as desired. 
\end{proof}

\section{Noether's nomarlization over division rings}

In this section, we introduce the notion \textit{``skew normalizable"}. Then we present our main result, Theorem \ref{main1}, which establishes the normalizability of the tuple $(\omega,\delta_1;\dots;\omega,\delta_n)$. We also prove a new version of the Combinatorial Nullstellensatz which plays a crucial role in our proof of the main result.

\begin{definition}
	Let $R\subseteq S$ be a ring extension. We say that $S$ is {\em  skew finitely generated} over $R$ if $S$ is isomorphic to a quotient ring $$R[t_1,\dots,t_n; (\omega_1,\delta_1),\dots,(\omega_n,\delta_n)]/I,$$ where
	\begin{itemize}
		\item [(1)] $\omega_1,\dots,\omega_n$ are automorphisms of $R$, and $\delta_i$ is an $\omega_i$-derivation of $R$ for $i=1,\dots,n$,
		\item [(2)] $I$ is a two-sided ideal of $R[t_1,\dots,t_n; (\omega_1,\delta_1),\dots,(\omega_n,\delta_n)]$ such that $I\cap R=\{0\}$.
	\end{itemize}
\end{definition}

\begin{proposition} \label{prop3.2}
	Let $R\subseteq S$ be a ring extension. Then $S$ is skew finitely generated over $R$ if and only if there exist automorphisms $\omega_1,\dots,\omega_m\in\Aut(R)$, and $\omega_i$-derivation $\delta_i$ of $R$  for $i=1,\dots,m$, and there are commuting elements $s_1,\dots,s_m\in S$ which are automorphic over $R$ with respect to $(\omega_1,\delta_1),\dots,(\omega_m,\delta_m)$ respectively such that $S=R[s_1,\dots,s_m]$. Sometimes, to emphasize, we also write this ring $R[s_1,\ldots,s_m]$ as $R[s_1,\ldots,s_m; (\omega_1,\delta_1),\dots,(\omega_m,\delta_m)]$.
\end{proposition}

\begin{proof}
	First, assume that for $i=1,\dots,m$, $\omega_i$ is an automorphism of $R$, $\delta_i$ is an $\omega_i$-derivation of $R$, and there are commuting elements $s_1,\dots,s_m\in S$ which are automorphic over $R$ with respect to $(\omega_1,\delta_1),\dots,(\omega_m,\delta_m)$ respectively. By Lemma \ref{eval2}, the evaluation map $$\mathrm{Eval}_{(s_1,\dots,s_n)}:  R[x_1,\dots,x_m; (\omega_1,\delta_1),\dots,(\omega_m,\delta_m)] \lra S$$
	in which $\mathrm{Eval}_{(s_1,\dots,s_m)}(f):=f(s_1,\dots,s_m)$ is a ring homomorphism. Since $s_1,\dots,s_m$ are commuting and $S=R[s_1,\dots,s_m]$, the map $\mathrm{Eval}_{(s_1,\dots,s_m)}(f)$ becomes an epimorphism, and thus $$S \cong R[x_1,\dots,x_m; (\omega_1,\delta_1),\dots,(\omega_m,\delta_m)]/\mathrm{Ker}(\mathrm{Eval}_{(s_1,\dots,s_m)})$$ 
	It is obvious that $\mathrm{Ker}(\mathrm{Eval}_{(s_1,\dots,s_m)})\cap R=\{0\}$, which implies that $S$ is  finitely generated over $R$.
	
	Conversely, assume that $S$ is skew finitely generated over $R$, which means $S=R[x_1,\dots,x_m; (\omega_1,\delta_1),\dots,(\omega_m,\delta_m)]/I$, in which $\omega_i\in \mathrm{Aut}(R)$ and $\delta_i$ is an $\omega_i$-derivation of $R$ for all $i=1,\dots,m$. We regard $R$ as a subring of $S$ by identifying $R$ with its image under the natural injection $R\hookrightarrow S$ where $r\mapsto r+I, \forall r\in R$. Let $s_i=x_i+I$ for $i=1,\dots,m$. Since $x_1,\dots, x_m$ are commuting, the elements $s_1,\dots,s_m \in S$ are also commuting. For all $r\in R$, $i=1,\dots,m$, we have $$s_ir=x_ir+I=\omega_i(r)x_i+\delta_i(r)+I=\omega_i(r)s_i+\delta_i(r)$$
	So $s_1,\dots,s_m$ are automorphic over $R$ with respect to $(\omega_1,\delta_1),\dots,(\omega_m,\delta_m)$ respectively. Finally, it is easy to see that $S=R[s_1,\dots,s_m]$.
\end{proof}

\begin{definition}
	Let $R \subseteq S$ be a ring extension. 
	\begin{enumerate}
		\item Commuting elements $s_1,\dots,s_n\in S$ are said to be {\em left algebraically independent} over $R$ if the set 
		$$\big\{ s_1^{i_1}\dots s_n^{i_n} \big\}_{i_1,\dots,i_n\in \mathbb{N}}$$ is left linearly independent over $R$.
		\item $S$ is called left {\em normalizable} over $R$ if there exist commuting elements $v_1,\dots, v_m\in S$ which are left algebraically independent over $R$  such that $S$ is a left finite module over $R[v_1,\dots,v_m]$.
	\end{enumerate}
	
\end{definition}

\begin{lemma}
	Let $R\subseteq S$ be a ring extension. Assume that
         commuting elements $s_1,\dots,s_m\in S$ form an iterated automorphic sequence over $R$ with respect to $(\omega_1,\delta_1),\dots,(\omega_m,\delta_m)$. Then $s_1,\dots,s_m$ are left algebraically dependent over $R$ if and only if there exists a non-zero polynomial $f\in R[x_1;\omega_1,\delta_1]\dots[x_m;\omega_m,\delta_m]$, where $x_1,\dots,x_m$ are commuting, such that $f(s_1,\dots,s_m)=0$.
\end{lemma}
\begin{proof}
Suppose that commuting elements $s_1,\dots,s_m\in S$ form an iterated automorphic sequence over $R$ with respect to $(\omega_1,\delta_1),\dots,(\omega_m,\delta_m)$. By Proposition \ref{eval1}, there exists a unique homomorphism $$\phi:R[x_1;\omega_1,\delta_1]\dots[x_n;\omega_n,\delta_n]\lra S$$ such that $\phi(x_i)=s_i$ for all $1\le i\le n$ and $\phi(r)=r$ for all $r\in R$.

If $s_1,\dots,s_m\in S$ are left algebraically dependent over $R$, then there is a finite set of  non-zero elements $A=\{a_{i_1,\dots,i_n}\in R, \text{ where } (i_1,\dots,i_n) \in  \mathbb{N}^n\}$ such that 
$$\sum_{a_{i_1,\dots,i_n}\in A}a_{i_1,\dots,i_n}s_1^{i_1}\dots s_n^{i_n}=0.$$
By denoting  $$f:=\displaystyle\sum_{a_{i_1,\dots,i_n}\in A}a_{i_1,\dots,i_n}x_1^{i_1}\dots x_n^{i_n}\in R[x_1;\omega_1,\delta_1]\dots[x_n;\omega_n,\delta_n],$$ we have that $f(s_1,\dots,s_n)=\phi(f)=0$. 

Conversely, assume that $f$ is a non-zero polynomial in $R[x_1;\omega_1,\delta_1]\dots[x_n;\omega_n,\delta_n]$ such that $f(s_1,\dots,s_n)=0.$ We can rewrite $f$ as
       $$\displaystyle f(x_1,\dots,x_n)=\sum_{I=(i_1,\dots,i_n)\in   \mathbb{N}^n}b_I x_1^{i_1}\dots x_n^{i_n},\ 0\ne b_I\in R.$$
 By applying the homomorphism $\phi$, we obtain that $$\phi(f)=\sum_{I=(i_1,\dots,i_n)\in \mathbb{N}^n}b_I s_1^{i_1}\dots s_n^{i_n}=f(s_1,\dots,s_n)=0.$$
Hence, $s_1,\dots,s_m\in S$ are left algebraically dependent over $R$.

\end{proof}

\begin{definition}\label{fngt}
Let $S$ be an extension ring of a division ring $D$. We say that $S$ is {\em skew normalizable} over $D$ if there exist, for $i=1,\dots,n$, an automorphism $\omega_i$ of $D$ and an $\omega_i$-derivation $\delta_i$ of $D$, and there exist commuting automorphic elements $s_1,\dots,s_n$  with respect to $(\omega_1,\delta_1),\dots,(\omega_n,\delta_n)$ respectively such that:
\begin{enumerate}
	\item $s_1,\dots,s_n$ are left algebraically independent over $D$.
	\item $S$ is finite as a left module over its subring $D[s_1,\dots,s_n]$.
\end{enumerate}
\end{definition}
\begin{remark}
In Definition \ref{fngt}, since $s_1,...,s_n$ are left algebraically independent over $D$, the kernel of the ring homomorphism $\phi:D[x_1,\dots,x_n; (\omega_1,\delta_1),\dots,(\omega_n,\delta_n)] \lra D[s_1,\dots,s_n]$, obtained from Proposition \ref{eval2}, is zero. So  $\phi$ is injective. On the other hand, it is obvious that $\phi$ is surjective. Thus,  the subring $D[s_1,\dots,s_n]$ of $S$ is isomorphic to the skew polynomial ring $D[x_1,\dots, x_n; (\omega_1,\delta_1),\dots,(\omega_n,\delta_n)]$. It follows that Definition \ref{fngt} and Definition \ref{def1} are equivalent.
\end{remark}

\begin{example}
    Let $\sigma\in\mathrm{Aut}(D)$. If $\sigma$ is an inner automorphism of $D$, then the ring $S=D[t,t^{-1};(\sigma,\delta),(\sigma^{-1},0)]$ is skew normalizable over $D$.
\end{example}
\begin{proof}
    Assume $\sigma(a)=cac^{-1}$ for all $a\in D$. Let $u=t+c^{2}t^{-1}\in S$. For each $a\in D$, we have
    \begin{align*}
        ua&=ta+c^2t^{-1}a=\sigma(a)t+\delta(a)+c^2\sigma^{-1}(a)t^{-1}=cac^{-1}t+c^2(c^{-1}ac)t^{-1}+\delta(a)\\
        &=(cac^{-1})(t+c^2t^{-1})+\delta(a)=\sigma(a)u+\delta(a).
    \end{align*}
Thus $u$ is automorphic over $D$ with respect to $(\sigma,\delta)$. Now we have $ut=t^2+c^2$, or $t^2-ut+c^2=0$. So $t$ is also left integral over $D[u;(\sigma,\delta)]$. We also have $ut^{-1}=1+c^2(t^{-1})^2$, or $(t^{-1})^2-c^{-2}ut^{-1}+1=0$. This means that $t^{-1}$ is left integral over $D$. Hence $S$ is finite over $D[u;(\sigma,\delta)]$, we conclude that $S$ is skew normalizable over $D$.
\end{proof}

\begin{definition}
	Let $D$ be a division ring. A tuple $(\omega_1,\delta_1;\dots;\omega_m,\delta_m)$, where $\omega_i\in \Aut(D)$ and $\delta_i$ is an $\omega_i$-derivation for $i = 1, \ldots, m$, is said to be {\em skew normalizable} over $D$ if every skew finitely generated extension of $D$ with respect to $(\omega_1,\delta_1),\dots,(\omega_m,\delta_m)$ is skew normalizable over $D$.
\end{definition}

Let $\omega$ be an automorphism of a ring $R$. We denote by $R_\omega$ the fixed subring of $R$ by $\omega$, i.e., $R_\omega=\{r\in R: \omega(r)=r\}$.
\begin{lemma}\label{mixderivation}
	Let $R$ be a ring and let $\omega\in\Aut(R)$. Suppose that $\delta_1,\dots,\delta_n$ are $\omega$-derivations of $R$ satisfying $\omega\circ\delta_i=\delta_i\circ\omega$ and $\delta_{i}\circ\delta_j=\delta_j\circ\delta_i$ for all $1\le i,j\le n$. Let $F=Z(R)\cap R_\omega\cap \mathrm{Ker}(\delta_1)\cap\cdots\cap\mathrm{Ker}(\delta_n)$, and $a_1,\dots,a_n\in F$. Assume that $s_1,\dots,s_n$ are commuting automorphic elements with respect to $(\omega,\delta_1),\dots,(\omega,\delta_n)$ respectively. Let $d_1=\delta_1+a_1\delta_n,\dots,d_{n-1}=\delta_{n-1}+a_{n-1}\delta_n,d_n=\delta_n$, and $u_1=s_1+a_1s_n,\dots,u_{n-1}=s_{n-1}+a_{n-1}s_n, u_n=s_n$. Then:
	\begin{enumerate}
{
\item $d_i$ is an $\omega$-derivation for each $1\leq i \leq n$.
}
		\item $\omega\circ d_i=d_i\circ\omega$ and $d_{i}\circ d_j=d_j\circ d_i$ for all $1\le i,j\le n$.
		\item $u_1,\dots,u_n$ are commuting automorphic over $R$ with respect to $(\omega,d_1),\dots,(\omega,d_n)$ respectively.
{
\item $F$ is contained in $Z(R)\cap R_\omega\cap \mathrm{Ker}(d_1)\cap\cdots\cap\mathrm{Ker}(d_n)$.
}
	\end{enumerate}
\end{lemma}
\begin{proof}
{
(1) Fix $1  \leq i \leq n-1$. For $a, b \in R$, we have
\begin{align*}
 d_i(ab) & = (\delta_i + a_i \delta_n)(ab)\\
 & = \delta_i(ab) + a_i \delta_n(ab)\\
 & = \omega(a) \delta_i(b) + \delta_i(a)b + a_i \omega(a) \delta_n(b) + a_i\delta_n(a)b \\
 & = [\omega(a) \delta_i(b)  + a_i \omega(a) \delta_n(b)] + [\delta_i(a)b + a_i\delta_n(a)b] \\
 & = \omega(a) [\delta_i(b)  + a_i \delta_n(b)] + [\delta_i(a) + a_i\delta_n(a)]b \\
  & = \omega(a) d_i(b) + d_i(a) b. 
\end{align*}
Hence, $d_i$ is an $\omega$-derivation.
}

(2) We have \begin{align*}
			\omega\circ d_i&=\omega\circ (\delta_i+a_i\delta_n)\\
			&=\omega\circ \delta_{i}+a_i\omega\circ \delta_{n}\\
			&=\delta_{i}\circ \omega+a_i\delta_n\circ\omega\\
			&=(\delta_i+a_i\delta_n)\circ\omega\\
			&=d_i\circ\omega.			
		\end{align*}
We also have
	\begin{align*}
		d_i\circ d_j&= (\delta_i+a_i\delta_n)\circ(\delta_j+a_j\delta_n)\\
		&=\delta_i\circ \delta_{j}+\delta_i\circ (a_j\delta_{n})+a_i\delta_n\circ\delta_j+a_i\delta_n\circ(a_j\delta_n)\\ &=\delta_i\circ\delta_j+a_j\delta_i\circ\delta_n+a_i\delta_j\circ\delta_n+a_ia_j\delta_n^2
	\end{align*}
    and
     \begin{align*}
     	d_j\circ d_i&=(\delta_j+a_j\delta_n)\circ(\delta_i+a_i\delta_n)\\
     		&=\delta_j\circ \delta_{i}+\delta_j\circ(a_i\delta_{n})+a_j\delta_n\circ\delta_i+a_j\delta_n\circ(a_i\delta_n)\\
     		&=\delta_j\circ\delta_i+a_i\delta_j\circ\delta_n+a_j\delta_i\circ\delta_n+a_ja_i\delta_n^2.
     \end{align*}
 So $d_i\circ d_j=d_j \circ d_i$.

\medskip

(3) We have 
     \begin{align*}
     	u_iu_j&=(s_i+a_is_n)(s_j+a_js_n)=s_is_j+a_js_is_n+a_is_ns_j+a_ia_js_ns_n\\
     	u_ju_i&=(s_j+a_js_n)(s_i+a_is_n)=s_js_i+a_is_js_n+a_js_ns_i+a_ja_is_ns_n
     \end{align*}
 and thus $u_iu_j=u_ju_i$. For all $r\in R$, we have for each $i = 1,\ldots,n-1$,
 \begin{align*}
u_ir&=s_ir+a_is_nr=\omega(r)s_i+\delta_i(r)+a_i\omega(r)s_n+a_i\delta_n(r)\\
 	 &=\omega(r)(s_i+a_is_n)+(\delta_i+a_i\delta_n)(r)\\
 	 &=\omega(r)u_i+d_i(r).
 \end{align*}
Also, $u_n r = \omega(r) u_n + \delta_n(r) = \omega(r) u_n + d_n(r).$
Therefore, $u_i$ is automorphic over $R$ with respect to $(\omega,d_i)$.

{
(4) It is obvious since $\mathrm{Ker}(\delta_1)\cap\cdots\cap\mathrm{Ker}(\delta_n) \subseteq \mathrm{Ker}(d_1)\cap\cdots\cap\mathrm{Ker}(d_n)$.
}
\end{proof}

\begin{example}
	Let $R=k[x_1,\dots,x_n]$ be the ring of polynomials over field $k$. Let $\omega$ be the identity map of $R$, choose $\delta_1=\frac{\partial}{\partial x_1}, \dots, \delta_n=\frac{\partial}{\partial x_n}$. It is easy to check that each $\delta_i$ is an $\omega$-derivation of $R$. Moreover, $\omega\circ\delta_i=\delta_i\circ\omega$ and $\delta_{i}\circ\delta_j=\delta_j\circ\delta_i$ for all $1\le i,j\le n$, and $F=Z(R)\cap R_\omega\cap \mathrm{Ker}(\delta_1)\cap\cdots\cap\mathrm{Ker}(\delta_n)=k$.
\end{example}

\begin{example}
Let $R=k[x_1,\dots,x_n]$ be the ring of polynomials over a field $k$ of characteristic 0. Let $e_1,\dots,e_n$ be non-zero elements in $k$. Let $\omega_1,\dots,\omega_n$ be automorphisms of $R$ defined by $\omega_i(x_i)=e_ix_i$ and $\omega_i(x_j)=x_j$ if $j\ne i$. For each $i=1,\dots,n$, assume $e_i\ne 1$, we define $\omega_i$-derivation as $\delta_i(f)=\dfrac{\omega_i(f)-f}{\omega_i(x_i)-x_i}$ for all $f\in R$. Then it is obvious that $\delta_i$ is an $\omega_i$-derivation. Moreover, $\omega_i$ commutes $\delta_j$, $\omega_i$ commutes  $\omega_j$,  $\delta_i$ commutes $\delta_j$ for $i\ne j$, and $\mathrm{Ker}(\delta_1)=\dots=\mathrm{Ker}(\delta_n)=R_{\omega_1}=\dots=R_{\omega_n}=k$. 
\end{example}

\begin{lemma}[Gordon-Motzkin  \cite{Gordon}]\label{gordon}
	Let $D$ be a division ring and $f(x)=a_nx^n+\dots+a_1x+a_0$ be a polynomial with left coefficients $a_0,\dots,a_n$ in $D$. Then at most $n$ conjugacy classes of $D$ contain roots of $f(x)$. 
\end{lemma}

{
 Note that if $f(t_1,\dots,t_n)$ is a  polynomial in the polynomial ring $D[t_1;,\dots,t_n]$, then $f$ may not be evaluated at an arbitrary element $(a_1,\dots,a_n)\in S^n$, where $S$ is a ring extension of $D$. However, by requiring $f(t_1,\dots,t_n)$ to be written in a ``standard" form, we can define the value of $f(t_1,\dots,t_n)$ at $(a_1,\dots,a_n)$   as follows.
 
 \begin{definition}\label{Eval3}
   Let $f(t_1,\dots,t_n)$ be a polynomial in $D[t_1,\dots,t_n]$, and let $S$ be a ring extension of $D$. For each  $ (a_1,\dots,a_n)\in S^n$, we define the value of $f(t_1,\dots,t_n)$ at $(a_1,\dots,a_n)$, denoted by $f(a_1,\dots,a_n)$, as follows: first we write $f(t_1,\dots,t_n)$ in the standard form
       $$\displaystyle f(t_1,\dots,t_n)=\sum_{I=(i_1,\dots,i_n)\in \mathbb{N}^n}b_I t_1^{i_1}\dots t_n^{i_n};$$
       then we set 
$$f(a_1,\dots,a_n):= \sum_{I=(i_1,\dots,i_n)\in \mathbb{N}^n}b_I a_1^{i_1}\dots a_n^{i_n}\in S.$$
 \end{definition}
 
\begin{remark}
    The evaluation at a point $(a_1,\dots,a_n)\in S^n$ in Definition \ref{Eval3} is a map from $D[t_1\dots,t_n]$ to $S$. But it is not necessary a ring homomorphism. However, from now on, we will use this notation without requiring it to be a ring homomorphism. 

\end{remark}
}

In the following Theorem, we present a new version of Combinatorial Nullstellensatz over division rings II, which was introduced in \cite[Theorem 3.2]{Paran1} by Elad Paran. Our proof is different from the one in \cite{Paran1}. As a consequence, we derive Corollary \ref{ComNulstellensatz2} which is analogous to \cite[Lemma 3.1]{Thieu1}.

\begin{theorem}[Combinatorial Nullstellensatz over division rings III]\label{ComNulstellensatz} Let $D$ be a division ring and $$\displaystyle f(t_1,\dots,t_n)=\sum_{I=(i_1,\dots,i_n)\in \mathbb{N}^n}b_I t_1^{i_1}\dots t_n^{i_n}$$ be a  polynomial of degree $m>0$ in $D[t_1\dots,t_n]$.
Let $A_1,\dots, A_n$ be subsets of $D$ such that the elements of each $A_i$ are {pairwise non-conjugate} and $|A_i|>m$ for $i=1,\dots,n$. Then there exists $(a_1,\dots,a_n)\in A_1\times\dots\times A_n$ such that $$f(a_1,\dots,a_n)= \sum_{I=(i_1,\dots,i_n)\in \mathbb{N}^n}b_I a_1^{i_1}\dots a_n^{i_n}\ne 0$$
\end{theorem}

\begin{proof}
The case $n=1$ holds true due to Lemma \ref{gordon}. Now consider $n>1$, suppose that $f(a_1, \ldots, a_n) =0$ for all $(a_1, \ldots, a_n) \in A_1 \times \cdots \times A_n.$
For any fixed $a_1\in A_1$, let  $$h(t_2,\dots,t_n)=f(a_1,t_2,\dots,t_n)=\sum_{I=(i_1,\dots,i_n)\in \mathbb{N}^n}b_I a_1^{i_1}t_2^{i_2}\dots t_n^{i_n}$$
Then $h(a_2,\dots,a_n)=f(a_1,a_2,\dots,a_n)=0$ for all $(a_2,\dots,a_n)\in A_2\times\dots\times A_n$. This implies, by induction, that $h(t_2,\dots,t_n)=0$.
Hence, $$f(a_1,t_2,\dots,t_n)=0,\ \ \forall a_1\in A.$$
By regarding $t_1$ as constant, we can consider $f$ as a polynomial in $D[t_1][t_2,\dots,t_n]$ and rewrite
$$\displaystyle f(t_1,\dots,t_n)=\sum_{J=(j_2,\dots,j_n) 
\in \mathbb{N}^{n-1}} f_J(t_1)t_2^{j_2}\dots t_n^{j_n}$$
where $f_J(t_1)\in D[t_1]$ and $\mathrm{deg}(f_J)\le m$. Now, for all $a_1\in A_1$, we have
$$0=f(a_1,t_2,\dots,t_n)=\sum_{J=(j_2,\dots,j_n)\in \mathbb{N}^{n-1}} f_J(a_1)t_2^{j_2}\dots t_n^{j_n}$$
So $f_J(a_1)=0$ for all $a_1\in A_1$. But $|A_1|>\mathrm{deg}(f_J)$, so by Lemma \ref{gordon} we obtain $f_J=0$, thus $f=0$, a contradiction.
\end{proof}

\begin{corollary}\label{ComNulstellensatz2}
	Let $D$ be a division ring and  $$\displaystyle f(t_1,\dots,t_n)=\sum_{I=(i_1,\dots,i_n)\in \mathbb{N}^n}b_I t_1^{i_1}\dots t_n^{i_n}$$ be a  polynomial of degree $m>0$ in $D[t_1,\dots,t_n]$.
	Let $A_1,\dots, A_n$ be subsets of the center $F=Z(D)$ such that $|A_i|>m$ for $i=1,\dots,m$. Then there exist $(a_1,\dots,a_n)\in A_1\times\dots\times A_n$ such that { $$f(a_1,\dots,a_n)=\sum_{I=(i_1,\dots,i_n)\in \mathbb{N}^n}b_I a_1^{i_1}\dots a_n^{i_n}\ne 0$$.}
\end{corollary}

\begin{proof}
	Assume that $a,b$ are two conjugate elements in $A_i$, then $a=xbx^{-1}$ for some $ x\in D$, hence $a=xx^{-1}b=b$, this means $A_i$ consists of non-conjugate elements. So, Theorem \ref{ComNulstellensatz} can be applied.
\end{proof}

\begin{lemma}\label{transivity}
	Let $R\subseteq S\subseteq T$ be a sequence of ring extensions. Suppose that $S$ is skew normalizable over $R$ and that $T$ is a left finite module over $S$. Then $T$ is also skew normalizable over $R$.
\end{lemma}

\begin{proof}
  The result follows by repeating the arguments in the proof of \cite[Lemma 3.6]{Thieu1}.
\end{proof}

\begin{theorem}\label{main1}
	Let $D$ be a division ring, $\omega$ be an automophism of $D$,  and $\delta_1,\dots,\delta_n$ be $\omega$-derivations of $D$ such that $\omega\circ\delta_i=\delta_i\circ\omega$ and $\delta_i\circ\delta_j=\delta_j\circ\delta_i$ for all $1\le i,j\le n$. Assume that $F=Z(D)\cap D_{\omega}\cap \mathrm{Ker}(\delta_1)\cap\cdots\cap \mathrm{Ker}(\delta_n)$ is an infinite set. Then the tuple $(\omega,\delta_1;\dots;\omega,\delta_n)$ is skew normalizable.
\end{theorem}

{
\begin{proof}
	Let $S$ be an arbitrary skew finitely generated extension of $D$. By Proposition \ref{prop3.2}, we write $S=D[x_1,\dots,x_n;(\omega,\delta_1),\dots,(\omega_,\delta_n) ]$. We will prove by induction on $n$ that $S$ is skew normalizable over $D$. 
	If $x_1,\dots, x_n$ are left  algebraically independent over $D$, then $S$ is skew normalizable. Now assume that $x_1,\dots,x_n$ are left  algebraically dependent over $D$, i.e. there is a non-zero polynomial in the skew polynomial ring $D[y_1,\dots, y_n;(\omega,\delta_1),\dots,(\omega_,\delta_n)]$ of the standard form $$f=\sum_{I=(i_1,\dots,i_n)\in \mathbb{N}^n}b_I y_1^{i_1}\dots y_n^{i_n},\ \ \ b_I\in D$$ such that $$f(x_1,\dots,x_n)=\sum_{I=(i_1,\dots,i_n)\in \mathbb{N}^n}b_I x_1^{i_1}\dots x_n^{i_n}=0.$$
For any $c_1,\dots,c_{n-1}\in F$, we rewrite 
\begin{align*}
    f(x_1,\dots,x_n)&=\sum_{I=(i_1,\dots,i_n)\in \mathbb{N}^n}b_I x_1^{i_1}\dots x_n^{i_n}\\
    &=\sum_{I=(i_1,\dots,i_n)\in \mathbb{N}^n}b_I (t_1+c_1t_n)^{i_1}\dots(t_{n-1}+c_{n-1}t_n) t_n^{i_n},
\end{align*}
where $t_1=x_1-c_{1}x_n,\dots,t_{n-1}=x_{n-1}-c_{n-1}x_n$, and $t_n=x_n$.
By our assumption, it is easy to check that $c_1,\dots,c_{n-1}$, $t_1,\dots,t_n$ commute with each other. Thus, for $k=1,\dots,n-1$, we have 
	\begin{align*}
		(t_k+c_kt_n)^{i_k}&=(c_kt_n)^{i_k}+\sum_{j=0}^{i_k-1}\binom{i_k}{j}t_k^{i_k-j}(c_kt_n)^j \\	
		&= c_k^{i_k}t_n^{i_k}+\sum_{j=0}^{i_k-1}u_{k,j}t_n^j,
	\end{align*}
    where $u_{k,j}=\binom{i_k}{j}c_k^jt_k^{i_k-j}\in F[t_k]$.
Hence,
\begin{align*}
	f(x_1,\dots,x_n)&=\sum_{\substack{I=(i_1,\dots,i_n)\in \mathbb{N}^n\\i_1+\dots+i_n\le N}}b_I\bigg(\prod_{k=1}^{n-1} (t_k+c_kt_n)^{i_k}\bigg) t_n^{i_n}, \ \ \text{ where } N=\mathrm{deg}(f)\\
	&=\sum_{\substack{I=(i_1,\dots,i_n)\in \mathbb{N}^n\\i_1+\dots+i_n\le N}}b_I\bigg(\prod_{k=1}^{n-1} \big(c_k^{i_k}t_n^{i_k}+\sum_{j=0}^{i_k-1}u_{k,j}t_n^j\big)\bigg) t_n^{i_n}\\
	&=\sum_{\substack{I=(i_1,\dots,i_n)\in \mathbb{N}^n\\i_1+\dots+i_n= N}}\big(b_I\prod_{k=1}^{n-1} c_k^{i_k})t_n^{N}+\sum_{j=0}^{N-1}v_jt_n^j, \text{ where } v_j\in F[t_1,\dots,t_{n-1}] \\
	&=h(c_1,\dots,c_{n-1})t_n^{N}+\sum_{j=0}^{N-1}v_jt_n^j,
\end{align*}
where $$h=\sum_{\substack{I=(i_1,\dots,i_n)\in \mathbb{N}^n\\i_1+\dots+i_n= N}}b_I y_1^{i_1}\dots y_{{n-1}}^{i_{n-1}}$$ is a polynomial with left coefficients in $D$, which is already in the standard form.

By applying Corollary \ref{ComNulstellensatz2}, we can choose specific elements $c_1,\dots, c_{n-1}\in F$ such that $h(c_1,\dots,c_{n-1})=a\ne0.$
Thus, we obtain that 
$$0=a^{-1}f(x_1,\dots,x_n)=t_n^{N}+\sum_{j=0}^{N-1}a^{-1}v_jt_n^j$$
Note that $a^{-1}v_j\in D[t_1,\dots,t_{n-1}]$, the subring of $S$ generated over $D$ by the elements $t_1,\dots,t_{n-1}$. 
Therefore, $t_n$ is left integral over $D[t_1,\dots,t_{n-1}]$. As a consequence, $S=D[t_1,\dots,t_{n-1}][t_n]$ is finitely generated as a left module over $D[t_1,\dots,t_{n-1}]$.

On the other hand, by Lemma \ref{mixderivation}, for each $k=1,\dots,n-1$, we have that $t_k=x_k-c_kx_n$ is automorphic over $D$ with respect to $(\omega,\delta_k-c_k\delta_n)$. Hence, $$D[t_1,\dots,t_{n-1}]=D[t_1,\dots,t_{n-1}; (\omega,\delta_1-c_1\delta_n),\dots,(\omega,\delta_{n-1}-c_{n-1}\delta_n)].$$

Now, by applying the induction hypothesis, we obtain that $D[t_1,\dots,t_{n-1}; (\omega,\delta_1-c_1\delta_n),\dots,(\omega,\delta_{n-1}-c_{n-1}\delta_n)]$ is skew normalizable over $D$. Thus, $S$ is skew   normalizable over $D$ due to Lemma \ref{transivity}.
\end{proof}
}

\begin{lemma}
    Let $\omega$ be an automorphism of a division ring $D$ and $\delta$ be an $\omega$-derivation of $D$. Let $F=Z(D)$ be the center of $D$. Suppose $\delta(F)\subseteq F$. Then $\omega(F)\subseteq F$ and $\delta$ is an $\omega$-derivation of $F$.
\end{lemma}
\begin{proof}
    It is sufficient to prove $\omega(F)\subseteq F$. Fix $u\in F$. For all $d\in D$, we have $\omega(ud)=\omega(du)$, so $\omega(u)\omega(d)=\omega(d)\omega(u)$. Thus, $\omega(u)\in F$.
\end{proof}

\begin{proposition}
    Let $D$ be a centrally finite division ring in its center $F=Z(D)$. Given commuting automorphisms $\omega_1,\dots,\omega_n$ in $\mathrm{Aut}(D)$ and $\omega_i$-derivation $\delta_i$ for $i=1,\dots,n$. Suppose that $\delta_i(F)\subseteq F$ for all $i=1,\dots,n$, and that the tuple $(\omega_1,\delta_1;\dots;\omega_n,\delta_n)$ is skew normalizable over $F$. If $S$ is skew finitely generated over $D$ with respect to $(\omega_1,\delta_1),\dots,(\omega_n,\delta_n)$, then $S$ is skew normalizable over $F$.
\end{proposition}
\begin{proof}
Suppose $S=D[x_1,\dots,x_n;(\omega_1,\delta_1),(\omega_n,\delta_n)]$ is a skew finitely  generated extension over $D$, where $x_i$ is automorphic over $D$ with respect to $(\omega_i,\delta_i)$. We prove that $S$ is skew normalizable over $F$.

From the assumption that $D$ is a finite module over $F$, we have that $S$ is a finite module over $F[x_1,\dots,x_n;(\omega_1,\delta_1),\dots(\omega_n,\delta_n)]$. Since the tube $(\omega_1,\delta_1;\dots;\omega_n,\delta_n)$ is skew normalizable over $F$, there exist automorphisms $\alpha_1,\dots,\alpha_s$ of $F$ and $\alpha_i$-derivation $\beta_i$ over $F$ for $i={1,..,s}$, and there are left algebraically independent elements over $F$ $$y_1,\dots,y_s\in F[x_1,\dots,x_n;(\omega_1,\delta_1),\dots(\omega_n,\delta_n)]$$ such that $F[x_1,\dots,x_n;(\omega_1,\delta_1),\dots(\omega_n,\delta_n)]$ is a finite module over $F[y_1,\dots,y_s;$ $(\alpha_1,\beta_1),\dots,(\alpha_s,\beta_s)]$. Hence, $S$ is a finite module over $F[y_1,\dots,y_s;(\alpha_1,\beta_1),\dots,$ $(\alpha_s,\beta_s)]$. By Lemma \ref{transivity}, $S$ is skew normalizable over $F$.
\end{proof}

\section*{Acknowledgements}
We would like to express our sincere gratitude for the time and effort that the referee has generously devoted to the thorough and meticulous review of our manuscript. His/her suggestions have greatly helped improve the presentation and the style of the paper. This research was funded by Ho Chi Minh city University of Technology and Engineering, Vietnam, under grant No. T2026-147.



\bibliographystyle{amsplain}

\end{document}